\documentclass{amsart}
\usepackage{amssymb,amsmath,latexsym}
\usepackage{amsthm}
\usepackage{fontenc}
\usepackage{amssymb}

\numberwithin{equation}{section}

\newtheorem{theorem}{Theorem}[section]

\begin{document}
\author{Alexander E Patkowski}
\title{A Generalized Davenport Expansion}

\maketitle
\begin{abstract} We prove a new generalization of Davenport's Fourier expansion of the infinite series involving the fractional part function over arithmetic functions. A new Mellin transform related to the Riemann zeta function is also established. \end{abstract}



\section{Introduction and Main Result} In 1937, Davenport presented infinite series over arithmetic functions, using the Fourier series of the fractional part function, $\{x\}=x-[x],$ where $[x]$ denotes the integer part of $x.$ The main result [4, eq.(2)] is the explicit formula,
\begin{equation}\sum_{n\ge1}\frac{a(n)}{n}\left(\{nx\}-\frac{1}{2}\right)=-\frac{1}{\pi}\sum_{n\ge1}\frac{A(n)}{n}\sin(2\pi n x).\end{equation}
Here $a(n)$ is an arithmetic function and $A(n)=\sum_{d|n}a(d).$ To date, many authors have researched (1.1) and associated identities, including its convergence [1, 2, 5, 6, 10, 11]. The principal idea of proving (1.1) through Mellin transforms can be found in Segal [11]. A number of authors have generalized (1.1) through use of periodic Bernoulli polynomials and Mellin inversion [2, 6]. 
\par The purpose of this article is to offer a new generalization of (1.1), and in accomplishing this we obtain a new Mellin transform. Our main theorem provides a Fourier series with coefficients for both sine and cosine, giving (1.1) as the special case $N=1.$ The Mellin transform given in our Theorem 1.2 is a more general form of an integral that has been used to obtain many interesting results, including the functional equation for the Riemann zeta function. Recall that the Riemann zeta function is $\zeta(s)=\sum_{n\ge1}n^{-s},$ for $\Re(s)>1.$ 

\begin{theorem}\label{thm:thm1} For $N\ge1,$ and $F_k(n):=\sum_{d|n}d^{-k}a(\frac{n}{d})$ we have for real $x,$
$$\sum_{n\ge1}\frac{a(n)}{n}\left(\{nx\}^{N}+N!\sum_{k\ge0}^{N-1}\frac{\zeta(-k)}{(N-k)!k!}\right)=$$
$$ -N!\sum_{k\ge0}^{N-1}\frac{(-1)^k}{(N-k)!}\left(\frac{\cos(\frac{\pi}{2}k)}{\pi}\sum_{n\ge1}\frac{F_k(n)}{n}\sin(2 \pi n x)-\frac{\sin(\frac{\pi}{2}k)}{\pi}\sum_{n\ge1}\frac{F_k(n)}{n}\cos(2 \pi n x)\right).$$
\end{theorem}
Our proof ensures that if the infinite series on one side of the theorem converges then the infinite series on the other side converges as well. The Mellin transform we need to establish Theorem 1.1 is given in the following, and fits in neatly with the family of integrals given in [3]. Let $(s)_k=\Gamma(s+k)/\Gamma(s)$ denote the Pochhammer symbol. \begin{theorem} For $0<\Re(s)<N,$ we have,

$$\int_{0}^{\infty}\{y\}^Ny^{-s-1}dy=N!\sum_{k\ge0}^{N-1}\frac{(-1)^k\zeta(s-k)}{(N-k)!(-s)_{k+1}}.$$
\end{theorem}

\begin{proof} A direct computation gives
\begin{equation}\begin{aligned}\int_{1}^{\infty}\{y\}^{N} y^{-s-1}dy
&=\sum_{k\ge1}\int_{k}^{k+1}\{y\}^N y^{-s-1}dy\\
&=\sum_{k\ge1}\int_{0}^{1}\frac{y^N}{(y+k)^{s+1}}dy \\
&=\int_{0}^{1}y^{N}\zeta(s+1,y+1)dy.\end{aligned}\end{equation}

From [7, pg.184, eq.(12.2)] we have 
\begin{equation}\int_{0}^{1}y^{N}\zeta(s,y)dy=N!\sum_{k\ge0}^{N-1}(-1)^k\frac{\zeta(s-k-1)}{(N-k)!(1-s)_{k+1}},\end{equation} where $\zeta(s,y)$ is the Hurwitz zeta function. The left side may be written \begin{equation}\int_{0}^{1}y^{N}\zeta(s,y+1)dy+\frac{1}{N-s+1}.\end{equation}
Note that, for $0<\Re(s)<N,$ 
\begin{equation}\int_{0}^{1}\frac{\{y\}^{N}}{y^{s+1}}dy=\int_{0}^{1}y^{N-s-1}dy=\frac{1}{N-s}.\end{equation}
Combining (1.2), (1.3), (1.4), and (1.5) gives the theorem.
\end{proof}

\section{Proof of Main Theorem}
To prove Theorem 1.1, we first obtain a different form of our main integral result contained in Theorem 1.2. After this is accomplished we generalize the proof of Segal [11].

\begin{proof}[Proof of Theorem~\ref{thm:thm1}] By Theorem 1.2,

\begin{equation}\{y\}^N=\frac{N!}{2\pi i}\int_{(c)}\left(\sum_{k\ge0}^{N-1}\frac{(-1)^k\zeta(s-k)}{(N-k)!(-s)_{k+1}}\right)y^sds,\end{equation}
if $0<\Re(s)=c<N.$
The integrand has a simple pole at $s=0.$ Note that
\begin{equation}\lim_{s\rightarrow0}\left(s\frac{\zeta(s-k)}{(-s)_{k+1}}y^{s}\right)=-\frac{\zeta(-k)}{\Gamma(k+1)}.\end{equation}
Therefore, computing the residue at the pole $s=0,$ and moving the line of integration to $-1<\Re(s)=d<0$ in (2.1),
\begin{equation}\frac{1}{2\pi i}\int_{(c)}\left(N!\sum_{k\ge0}^{N-1}\frac{(-1)^k\zeta(s-k)}{(N-k)!(-s)_{k+1}}\right)y^sds\end{equation}
$$=N!\sum_{k\ge0}^{N-1}\frac{\zeta(-k)}{(N-k)!k!}+\frac{1}{2\pi i}\int_{(d)}\left(N!\sum_{k\ge0}^{N-1}\frac{(-1)^k\zeta(s-k)}{(N-k)!(-s)_{k+1}}\right)y^sds.$$
Collectively, we have for $y>0$
\begin{equation}\{y\}^N-N!\sum_{k\ge0}^{N-1}\frac{\zeta(-k)}{(N-k)!k!}\end{equation}
$$=\frac{1}{2\pi i}\int_{(d)}\left(N!\sum_{k\ge0}^{N-1}\frac{(-1)^k\zeta(s-k)}{(N-k)!(-s)_{k+1}}\right)y^sds,$$
Inverting the desired series over the coefficients $a(n)$ in (2.4), we have
\begin{equation}\sum_{n\ge1}\frac{a(n)}{n}\left(\{ny\}^N-N!\sum_{k\ge0}^{N-1}\frac{\zeta(-k)}{(N-k)!k!}\right) \end{equation}
$$=N!\sum_{k\ge0}^{N-1}\frac{(-1)^k}{(N-k)!}\frac{1}{2\pi i}\int_{(d)}\frac{\zeta(s-k)}{(-s)_{k+1}}L(1-s)y^{s}ds,$$

Notice that, by the functional equation for the Riemann zeta function [12, pg.13, Theorem 2.1],
$$\begin{aligned}\frac{\zeta(s-k)L(1-s)y^{s}}{(-s)_{k+1}} &=\frac{\sin(\frac{\pi}{2}(s-k))\Gamma(1+k-s)\zeta(1+k-s)L(1-s)\Gamma(-s)y^s}{\Gamma(k+1-s)} \\
&=\Gamma(-s)\sin(\frac{\pi}{2}(s-k))\zeta(1+k-s)L(1-s)y^{s},\end{aligned}$$
and $\sin(\frac{\pi}{2}(s-k))=\sin(\frac{\pi}{2}s)\cos(\frac{\pi}{2}k)+\cos(\frac{\pi}{2}s)\sin(\frac{\pi}{2}k).$ 
Note that $\zeta(k+s)L(s)=\sum_{n\ge1}F_k(n)n^{-s},$ for $\Re(s)>1.$ Therefore, replacing $s$ by $-s$ in our integral in (2.5), and employing [8, pg.406]
$$\int_{0}^{\infty}y^{s-1}\cos(2\pi y)dy=(2\pi)^{-s}\Gamma(s)\cos(\frac{\pi}{2}s),$$
$$\int_{0}^{\infty}y^{s-1}\sin(2\pi y)dy=(2\pi)^{-s}\Gamma(s)\sin(\frac{\pi}{2}s),$$
both valid for $0<\Re(s)<1,$ we obtain the Fourier series in the right hand side in Theorem 1.1. 
\end{proof}

1390 Bumps River Rd. \\*
Centerville, MA
02632 \\*
USA \\*
E-mail: alexpatk@hotmail.com, alexepatkowski@gmail.com

\end{document}